\documentclass[10pt]{article}
\usepackage{latexsym,fancyhdr}

\usepackage{amssymb,amsmath,amsthm,array,eufrak}
\usepackage{dsfont,graphicx}
\usepackage{pst-plot,pstricks,multido}
\usepackage{mathrsfs}

\newtheorem{thm}{Theorem}[section]
\newtheorem{lem}[thm]{Lemma}
\newtheorem{cor}[thm]{Corollary}

\newtheorem{prop}[thm]{Proposition}
\theoremstyle{definition}
\newtheorem{defn}{Definition}

\newtheorem{exam}{Example}
\newtheorem*{remk}{Remark}

\newcommand{\N}{\mathbb{N}}
\newcommand{\Z}{\mathbb{Z}}
\newcommand{\Q}{\mathbb{Q}}
\newcommand{\R}{\mathbb{R}}
\newcommand{\norm}[1]{\left\Vert#1\right\Vert}
\newcommand{\abs}[1]{\left\vert#1\right\vert}
\newcommand{\set}[1]{\left\{#1\right\}}
\newcommand{\paren}[1]{\left(#1\right)}
\newcommand{\brac}[1]{\left[#1\right]}

\newcommand{\iprod}[1]{\left\langle#1\right\rangle}

\newcommand{\one}{\mathds{1}}

\DeclareMathOperator{\Supp}{supp}
\DeclareMathOperator{\graph}{graph}

\newcommand{\reftitle}[1]{#1.}
\newcommand{\refvol}[1]{#1,}
\newcommand{\draftnote}[1]{}

\newcommand{\commentout}[1]{}

\newcommand{\deletedProof}[1]{}

\begin{document}
\title{Supports of Implicit Dependence Copulas}
\author{Songkiat Sumetkijakan\\
Department of Mathematics and Computer Science\\ 
Faculty of Science, Chulalongkorn University\\
Phayathai Road, Bangkok, Thailand}

\date{June 24, 2016}
\maketitle

\begin{abstract}
A copula of continuous random variables $X$ and $Y$ is called an \emph{implicit dependence copula} if there exist functions $\alpha$ and $\beta$ such that $\alpha(X) = \beta(Y)$ almost surely, which is equivalent to $C$ being factorizable as the $*$-product of a left invertible copula and a right invertible copula.  Every implicit dependence copula is supported on the graph of $f(x) = g(y)$ for some measure-preserving functions $f$ and $g$ but the converse is not true in general.  

We obtain a characterization of copulas with implicit dependence supports in terms of the non-atomicity of two newly defined associated $\sigma$-algebras.  As an application, we give a broad sufficient condition under which a self-similar copula 
has an implicit dependence support.   
Under certain extra conditions, we explicitly compute the left invertible and right invertible factors of the self-similar copula.

%
\end{abstract}



\section{Introduction} 

It is well-known that the bivariate copula of two random variables 
completely captures their dependence structure.  
Notable examples are the \emph{independence copula} $\Pi(u,v) = uv$, which corresponds to independent random variables, and the copulas of completely dependent random variables, called \emph{complete dependence copulas}.  
Since it was discovered \cite{MikSheTay1992-Shuffles, MikSheTay1991-Probabilistic, KimSam1978-Monotone}  
 that there are complete dependence copulas arbitrarily closed to $\Pi$ in the uniform norm, many norms have been introduced and investigated in the literature \cite{DarOls1995-Norms} giving rise to measures of dependence such as 
$\omega$ in \cite{SibSto2009-Measure} and $\zeta_1$ in \cite{Tru2011-Strong}.
%
These dependence measures defined in terms of the copula's first partial derivatives 
attain the maximum value $1$ at least for complete dependence copulas and the minimum value $0$ when and only when the copula is $\Pi$. 
%
However, with respect to the measure of mutual complete dependence (MCD) $\omega$, the independence copula can still be approximated by \emph{implicit dependence copulas} \cite{ChaSanSum2016-Patched}, defined as copulas of random variables $X$ and $Y$ which are implicitly dependent in the sense that $f(X) = g(Y)$ a.s.~for some Borel measurable functions $f$ and $g$.  
%
%
For R\'enyi-type measures of dependence \cite{Ren1959-Measures} such as $\omega_*$ in \cite{RuaSanSum2013-Shuffles} and $\nu_*$ in \cite{KamSanSum2015-Dependence}, with respect to which all complete dependence copulas have measure $1$, it can be proved that all implicit dependence copulas also have maximum measure $1$. 
It is then evident that implicit dependence copulas plays a crucial role in understanding as well as comparing and contrasting measures of MCD and R\'enyi-type dependence measures.
Closely related and constituting a much larger class than the implicit dependence copulas are the copulas whose supports are that of an implicit dependence copula.
%

We shall investigate 
copulas with implicit dependence supports via their corresponding Markov operators and associated pairs of $\sigma$-algebras.
Our approach is motivated by the characterization of idempotent copulas via their $\sigma$-algebras of invariant sets in \cite{DarOls2010-Characterization}.  
In particular, they proved that non-atomic idempotent copulas must be of the form $L*L^t$ for some left invertible copula $L$ (where $C^t$ denotes the transpose of $C$, i.e.~$C^t(x,y) = C(y,x)$, and $*$ is the product of copulas first defined and studied in \cite{DarNguOls1992-Copulas, OlsDarNgu1996-Copulas}).  
We define $\sigma$-algebras $\sigma_C$ and $\sigma_C^*$ associated with a copula $C$ for which the corresponding Markov operator $T_C$ maps indicator functions of sets in $\sigma_C$ to indicator functions of sets in $\sigma_C^*$.  
%
We prove that copulas with implicit dependence supports are exactly copulas whose both $\sigma$-algebras are non-atomic.  
Our main result finds an application in copulas with fractal supports introduced by Fredricks, Nelsen and Rodr\'iguez-Lallena \cite{FreNelRod2005-Copulas}.  
Given a transformation matrix $A$, there is a unique \emph{self-similar copula} $C_A$ such that $[A](C_A) = C_A$, where $[A]$ maps the class 
of bivariate copulas into itself according to the weights given by the entries in $A$. 
As a consequence, we obtain a broad sufficient condition on a transformation matrix $A$ under which the self-similar copula $C_A$ is non-atomic and hence has an implicit dependence support.  
Working directly with the transformation matrix $A$, a sufficient condition under which $C_A$ is an implicit dependence copula is also given.
It would be even more interesting if we had a characterization of implicit dependence copulas via behaviors of their $\sigma$-algebras.  
We are hopeful that our future attempts will not be futile as such a characterization would be beneficial in studies involving products of implicit dependence copulas.

%
%
%
%

The manuscript is organized as follows.  Section \ref{section:background} lays the necessary background on copulas and Markov operators for the rest of the paper.  We then define the associated $\sigma$-algebras of a copula and prove their basic properties in Section \ref{section:sigma-algebras}.  Section \ref{section:non-atomic} gives a definition of non-atomic copulas and some of their fundamental properties summarizing in a characterization of non-atomic copulas.  In the final section, the characterization is used in an investigation of copulas with fractal support.  We also give a sufficient condition on a transformation matrix under which the induced invariant copula can be written as the product of a left invertible copula and a right invertible copula.

\section{Background on copulas and Markov operators} \label{section:background}
Let $\lambda$ denote the Lebesgue measure on $\R$, $I \equiv [0,1]$ and $\mathscr{B} \equiv \mathscr{B}(I)$ the Borel $\sigma$-algebra on $I$.   
Since we always consider $\lambda$-integrable functions on $I$ that are measurable with respect to various sub-$\sigma$-algebras $\mathscr{M}$ of $\mathscr{B}$, we will denote by $L^1(\mathscr{M})$ 
the class of $\lambda$-integrable $\mathscr{M}$-measurable functions on $I$. 
$\mathscr{B}$-measurable functions are called \emph{Borel functions}.
 For $A\in \mathscr{B}$, $\one_A$ denotes the indicator function of $A$ and $\one \equiv \one_I$.

A (bivariate) \emph{copula} $C$ is a function from $I^2$ to $I$ which is the joint distribution function of two random variables uniformly distributed on $[0,1]$.  
For random variables $X,Y$ with joint distribution $F_{X,Y}$ and continuous marginal distributions $F_X$ and $F_Y$, there exists, by the Sklar's theorem, a unique copula $C$, called the \emph{copula of $X$ and $Y$}, for which $F_{X,Y}(x,y) = C(F_X(x),G_Y(y))$ for all $x,y$.  
The \emph{independence copula} is the product copula $\Pi(u,v) = uv$.  
\emph{Complete dependence copulas} are either the copulas $C_{ef}\equiv C_{e,f}$ or $C_{fe}\equiv C_{f,e}$ where $e(x) = x$ and $f$ is a measure-preserving function on $I$ in the sense that $\lambda(f^{-1}(B)) = \lambda(B)$ for all $B\in\mathscr{B}$.  
Here, $C_{f,g}(u,v) = \lambda(f^{-1}([0,u])\cap g^{-1}([0,v]))$ for $u,v \in I$.  
The \emph{comonotonic} and \emph{countermonotonic} copulas are $M = C_{e,e}$ and $W = C_{e,1-e}$, respectively.  
Each copula $C$ induces a \emph{doubly stochastic measure} $\mu_C$ by $\mu_C((a,b]\times (c,d]) = C(b,d) - C(b,c) - C(a,d) + C(a,c)$.  
The \emph{support} of $C$ is then defined as the support of the induced measure $\mu_C$, i.e.~the complement of the union of all open sets having zero $\mu_C$-measure.
One can construct a new copula by taking any convex combinations of two copulas.
Any two copulas $C,D$ also give rise to a new copula via the \emph{$*$-product}: $(C*D)(u,v) = \int_0^1 \partial_2C(u,t) \partial_1 D(t,v) \,dt$.  
The binary operation $*$ makes the class of copulas a monoid with null element $\Pi$ and identity $M$.
If $C*D = M$ then $C$ is a left inverse of $D$ and $D$ is a right inverse of $C$. 
The left invertible copulas are exactly the complete dependence copulas $C_{ef}$, while the right invertible copulas are exactly the complete dependence copulas $C_{fe}$.  
See \cite{Nel2006-Introduction, DurSem2015-Principles} for comprehensive introductions to many aspects of copulas.

A linear operator $T$ on $L^1(\mathscr{B})$ is called a \emph{Markov operator} if 
\begin{enumerate} 
\renewcommand{\theenumi}{\roman{enumi}}
  \item $T\mathds{1}=\mathds{1}$, 
  \item \label{ii} $\int_0^1 T\psi\,d\lambda = \int_0^1 \psi\,d\lambda$ for all $\psi\in L^1$, which is equivalent to $T^*\mathds{1} = \mathds{1}$, and
  \item $T\psi \geq 0$ for all $\psi \geq 0$, which means $T$ is positive.
\end{enumerate}
So a Markov operator must be a bounded linear operator on $L^1$ (and $L^\infty$). 
From \cite{OlsDarNgu1996-Copulas}, for each copula $C$, there corresponds a Markov operator $T_C$ defined by
\[(T_C\psi)(x) = \dfrac{d}{dx} \int_0^1\partial_2C(x,t)\psi(t)\,dt.\]
In fact, the mapping $\Phi\colon C\mapsto T_C$ is an isomorphism 
from the set of copulas endowed with the $*$-product onto the set of Markov operators 
under the composition.  
In particular,  $T_{C*D} = T_C\circ T_D$. 
Define also $(T_C^*\varphi)(y) = \frac{d}{dy} \int_0^1\partial_1C(t,y)\varphi(t)\,dt.$
The copula $C$ and the Markov operators $T_C$ and $T_C^*$ are also related by the identities
\begin{equation}
 C(x,y) = \int_0^x T_C\mathds{1}_{[0,y]} (t)\,dt \quad\text{and}\quad C(x,y) = \int_0^y T_C^*\mathds{1}_{[0,x]} (t)\,dt.
\end{equation} 
In fact, it is a good exercise in functional analysis to verify that the Markov operator $T_C^*$ coincides with the extension of the adjoint of $T_C$ to a Markov operator on $L^1$, i.e.~$T_C^* = (T_C)^*$.  
%

Let us quote a very useful result from \cite[Theorem 2.11]{DarOls2010-Characterization} where, for brevity, we denote $T_{C_{f,g}} = T_{f,g}$.
\begin{thm} \label{thm:2.11DO2010}
 Let $f$ be a measure-preserving Borel function and $\psi \in L^1(\mathscr{B})$. Then $[T_{ef} \psi](x) = \psi \circ f (x)$ and $[T_{fe}(\psi\circ f)](x) = \psi(x)$ for almost all $x \in [0,1]$. 
\end{thm}


\section{Associated $\sigma$-algebras} \label{section:sigma-algebras}
Unless stated otherwise, all equalities of two functions hold $\lambda$-almost everywhere and all equalities of two sets mean that their symmetric difference has Lebesgue measure zero. The integral on the whole unit interval $I$ is denoted simply by $\int$.

Let $C$ be the copula of 
random variables $X$ and $Y$ which are uniformly distributed on $[0,1]$. 
Then for any Borel sets $R,S \subseteq [0,1]$, we have 
\[T_C\mathds{1}_S(x) = P(Y\in S| X=x)\quad\text{and}\quad T^*_C\mathds{1}_R(y) = P(X\in R| Y=y).\]
They are  called \emph{transition probabilities}.  See \cite{DarNguOls1992-Copulas} 
\draftnote{Pf of 2nd eq?  Easy for abs.cont.~$C$ (by def.~of adj.)} 
We also have $T^*_C = T_{C^t}$ and  $\paren{T^{*}_C}^* = T_{C}$.  \draftnote{Proof?} 
Roughly speaking, if $T_C\mathds{1}_S = \one_R$, then it happens with probability one that if $X \in R$ then $Y\in S$.
For each copula $C$ or Markov operator $T = T_C$, let us define 
\begin{align*}
  \sigma_C = \sigma_T &\equiv \set{S\in\mathscr{B}\colon T\mathds{1}_S = \mathds{1}_R \text{ for some }R}\;\text{ and }\\
  \sigma^*_C = \sigma^*_T &\equiv \set{R\in\mathscr{B}\colon T\mathds{1}_S = \mathds{1}_R \text{ for some }S}.
\end{align*}

\begin{exam} \label{exam:sigma}
Let us explicitly compute $T_C$ for $C = \Pi$, $C=$ some complete dependence copulas and $C = \frac{M+W}{2}$.
\begin{enumerate}
  \item $T_\Pi\one_S = \one_R$ is equivalent to $\displaystyle \mathds{1}_R(x) = \dfrac{d}{dx} \int_S\partial_2\Pi(x,t)\,dt = \lambda(S)$ for a.e.~$x\in [0,1].$
  So $\lambda(S) = 0$ or $1$ and hence $\lambda(R)=\lambda(S)=0$ or $1$. Thus, $\sigma_\Pi = \sigma^*_\Pi = \set{S\in\mathscr{B}\colon \lambda(S) = 0 \text{ or } 1}$.  
  
  \item  With essentially the same arguments as that in \cite{DarOls2010-Characterization}, $\sigma_M$, $\sigma^*_M$, $\sigma_W$ and $\sigma^*_W$ are the Borel sets.
  
  \item  For $C=\frac{M+W}{2}$, $T_C\one_S(x) = \frac{1}{2}\mathds{1}_S(x) + \frac{1}{2}\mathds{1}_S(1-x)$ a.e.~$x$.  If $S$ is symmetric with respect to $\frac{1}{2}$, i.e.~$x\in S$ if and only if $1-x \in S$, then $T_C\one_S = \one_S$.  Conversely, $T_C\one_S(x) = \mathds{1}_R(x)$ implies that $x\in S$ if and only if $1-x\in S$. 
Moreover, for such a symmetric set $S$, $(\one_R - \one_S)(x) = (T_C\one_S - \one_S)(x) = \frac{1}{2}\paren{\one_S(1-x) - \one_S(x)} = 0$ a.e.~$x$.  That is, $R=S$.  Hence, $\sigma_{\frac{M+W}{2}} = \sigma^*_{\frac{M+W}{2}} = \set{S\subseteq [0,1]\colon x\in S \Leftrightarrow 1-x \in S}$. 
%
  
  \item For $0<\alpha<1$, let $L_\alpha$ denote the complete dependence copula whose support consists of the line segments $y = \frac{x}{\alpha}$, $0\leq x\leq \alpha$, and $y = \frac{x-\alpha}{1-\alpha}$, $\alpha < x\leq 1$.  Then a direct computation yields $T_{L_\alpha} \one_S = \one_{(\alpha S) \cup (\alpha + (1-\alpha)S)}$ and hence  $\sigma_{L_\alpha} = \mathscr{B}$ and $\sigma_{L_\alpha}^* = \set{ \alpha S \cup \paren{\alpha + (1-\alpha)S}\colon S\in\mathscr{B}}$.
\end{enumerate}
\end{exam}

Listed below 
 are basic properties of sets $S$ and $R$ linked by $T$.
\begin{prop} \label{prop:sigma} 
  Let $T$ be a Markov operator and $S_1,S_2,R_1,R_2 \in \mathscr{B}$.  Then
  \begin{enumerate}
    \item \label{T**=T} if $T\mathds{1}_{S_1} = \mathds{1}_{R_1}$ then $T^*\one_{R_1} = \one_{S_1}$;
  \draftnote{
  $T^{**}=T$?}
    \item \label{sigmaminus} if $T\one_{S_1} = \one_{R_1}$ then $T\one_{S_1^c} = \one_{R_1^c}$;  
    \item \label{sigmacap} if $T\one_{S_i} = \one_{R_i}$ for $i=1,2$, then $T\one_{S_1\cap S_2} = \one_{R_1\cap R_2}$; 
    \item \label{sigma} the classes $\sigma_T$ and $\sigma^*_T$ are $\sigma$-algebras; and   
    \item \label{sigma*} $\sigma^*_T = \sigma_{T^*}$, that is $\sigma^*_C = \sigma_{C^t}$. \draftnote{Proof?} 
  \end{enumerate} 
\end{prop}
\begin{proof}
\textit{\ref{T**=T}}: By the definition of $T^*\one_R$ applied to $\one_S$ with the canonical identification between $\paren{L^1}^*$ and $L^\infty$, 
$ \int_S T^*\one_R\,d\lambda = \int \one_S T^*\one_R\,d\lambda = \int \one_R \one_R \,d\lambda = \lambda(R) = \int_S \one_S \,d\lambda$.  
Since $0\leq T^*\one_R \leq 1$ (a.e.), $T^*\one_R = \one_S$ as desired.
\textit{\ref{sigmaminus}}: If $T\mathds{1}_S = \mathds{1}_R$ then it follows from 
  $T\mathds{1}_S + T\mathds{1}_{S^c} = T(\mathds{1}_S + \mathds{1}_{S^c}) = T\mathds{1} = \mathds{1} = \mathds{1}_R + \mathds{1}_{R^c}$
that $T\mathds{1}_{S^c}=\mathds{1}_{R^c}$. 
%
 \textit{\ref{sigmacap}}: Suppose $T\mathds{1}_{S_i} = \mathds{1}_{R_i}$ for $i=1,2$ and let $S=S_1\cap S_2$, $R=R_1\cap R_2$.  
  Since $S\subseteq S_i$, $\mathds{1}_{S_i}-\mathds{1}_S \geq 0$ and so $\mathds{1}_{R_i}-T\mathds{1}_S = T(\mathds{1}_{S_i}-\mathds{1}_S) \geq 0.$  Therefore $T\mathds{1}_S \leq \min(\mathds{1}_{R_1},\mathds{1}_{R_2}) = \mathds{1}_{R_1\cap R_2} = \mathds{1}_R$.  But $\int \paren{\mathds{1}_R - T\mathds{1}_S}\,d\lambda = \int \mathds{1}_R \,d\lambda - \int \mathds{1}_S \,d\lambda = \lambda(R) - \lambda(S) = 0,$ where the last equality follows from considering $T^*\one_{R_2}(\one_{S_1}) = \one_{R_2}(T\one_{S_1})$. 
  Hence, $T\mathds{1}_S = \mathds{1}_R$.
 \textit{\ref{sigma}}:  Suppose $T\one_{S_i} = \one_{R_i}$ for all $i\in\N$.  
  By \ref{sigmacap}., if $S_1, S_2, \dots$ are mutually disjoint 
  then so is the sequence $R_1, R_2, \dots$ because 
  $\sum_{i=1}^\infty \mathds{1}_{R_i} = \sum_{i=1}^\infty T\mathds{1}_{S_i} = T\paren{\sum_{i=1}^\infty \mathds{1}_{S_i}} = T\paren{\mathds{1}_{\dot{\cup}_i S_i}} \leq \mathds{1}.$  
Thus, $T\paren{\mathds{1}_{\dot{\cup}_i S_i}} = \one_{\dot{\cup}_i R_i}$.
  Generally, we write $\bigcup_i S_i$ as the disjoint union $\cup_i \tilde{S}_i$, where $\tilde{S}_i = S_i\setminus \cup_{j<i} \tilde{S}_j$.  Letting $T\one_{\tilde{S}_i} = \one_{\tilde{R}_i}$, it follows from \ref{sigmaminus}., \ref{sigmacap}.~and the disjoint union case above that $T\mathds{1}_{S_i\setminus \cup_{j<i} \tilde{S}_j} = \one_{R_i\setminus \cup_{j<i} \tilde{R}_j}$.  Hence, $\tilde{R}_i = R_i\setminus \cup_{j<i} \tilde{R}_j$ by induction.  Consequently, $T\paren{\mathds{1}_{\cup_i S_i}} = \one_{\cup_i R_i}$.
 \textit{\ref{sigma*}}:  This clearly follows from \ref{T**=T}.~and the fact that $T^{**}=T$. \draftnote{Really?}
\end{proof}
\begin{remk}
  By Theorem \ref{thm:2.11DO2010}, $T_L^*\circ T_L = I$, the identity map, on $L^1(\mathscr{B})$ if the copula $L$ is left invertible.  However, it follows from the above Proposition that for \emph{all copula} $C$, $T_C^*\circ T_C = I$  on $\set{\one_S\colon S\in\sigma(C)}$, or equivalently on the class of $\sigma(C)$-measurable functions.  
\end{remk}


%
%
%

\begin{thm}
  Let $T$ be a Markov operator with associated $\sigma$-algebras $\sigma_T$ and $\sigma^*_T$.  
  If $\psi$ is $\sigma_T$-measurable then $T\psi$ is $\sigma^*_T$-measurable. \draftnote{Converse?}
\end{thm}
\begin{proof}
  By the linearity of $T$ and the definition of $\sigma_T$ and $\sigma^*_T$, if $\psi$ is a simple $\sigma_T$-measurable function then $T\psi$ is simple and $\sigma^*_T$-measurable.  The case when $\psi$ is $\sigma_T$-measurable follows from the continuity of $T$.
\end{proof}

%
%
%

\section{Non-atomic copulas} \label{section:non-atomic}
\begin{defn}
  Let $\mathscr{S}$ be a sub-$\sigma$-algebra of $\mathscr{B}$.
  A set $S$ in $\mathscr{S}$ is called an \emph{atom} in $\mathscr{S}$ if $\lambda(S) > 0$ and for all $E\in \mathscr{S}$, either $\lambda(S\cap E) = \lambda(S)$ or $\lambda(S\cap E) = 0$.  The $\sigma$-algebra $\mathscr{S}$ is said to be \emph{non-atomic} if there are no atoms in $\mathscr{S}$; otherwise, it is called \emph{atomic}.  $\mathscr{S}$ is \emph{totally atomic} if there is a (countable) collection of essentially disjoint atoms $E_1, E_2, ...$ in $\mathscr{S}$ such that $\sum_i \lambda(E_i) = 1$. 
We say that a bivariate copula $C$ is \emph{non-atomic} if both $\sigma$-algebras $\sigma_C$ and $\sigma^*_C$ are non-atomic. 
And $C$ is called \emph{totally atomic} if both $\sigma_C$ and $\sigma^*_C$ are totally atomic. 
\end{defn}
Note that the non-atomicity is a generalization of the notion of the same name in \cite{DarOls2010-Characterization}, which is defined only for idempotent copulas $C$ via their invariant sets defined as Borel sets $S$ for which $T_C\one_S = \one_S$. However, the two notions agree for idempotent copulas.
\begin{prop}
  If a copula $C$ is non-atomic and idempotent then $\sigma_C = \sigma_C^*$ is the $\sigma$-algebra of invariant sets.
\end{prop}
\begin{proof}
 Since idempotent copula is symmetric, $T^*_C = T_{C^t} = T_C$.  
  Consequently, $T^*_C\circ T_C = T_C\circ T_C = T_{C*C} = T_C$.  Now, if $T_C\one_S = \one_R$ then $T^*_C\one_R = \one_S$ and so $T^*_C\circ T_C \one_S = \one_S$.  Hence $S$ is an invariant set of $C$ and so is $R$ as $\one_R = T_C\one_S = \one_S$.  That is, $\sigma_C$ and $\sigma_C^*$ are subsets of the class of invariant sets. The converse inclusions are clear.
\end{proof}

\begin{exam} \label{exam:atomic}
It is evident from the computations of $\sigma_C$ and $\sigma^*_C$ in Example \ref{exam:sigma} that $\Pi$ is totally atomic but $M, W, \frac{M+W}{2}$ and $L_\alpha$ are non-atomic.   In fact, every complete dependence copula is non-atomic.
%
%
%
\end{exam}


%



\begin{prop}
  Let $C$ be a copula.
  \begin{enumerate}
    \item \label{sigma-left} $\sigma_C = \mathscr{B}$ 
     if and only if 
    $C$ is left invertible. 
    \item \label{sigma-right} $\sigma^*_C = \mathscr{B}$ 
     if and only if 
    $C$ is right invertible. 
  \end{enumerate}
\end{prop}
\begin{proof}
We shall prove only \ref{sigma-left}.~as \ref{sigma-right}.~can be proved in a similar manner.  
If $C = C_{eg}$ then $T_C = T_{eg}$ maps $\psi$ to $\psi\circ g$. So for every $B\in\mathscr{B}$, $T_C\one_B = \one_B\circ g = \one_{g^{-1}(B)}$. Hence, $\sigma_C=\mathscr{B}$. 
Conversely, 
 Theorem 5 in 
  \cite{SakShi1972-Equimeasurability} implies in particular that $T$ is multiplicative, meaning $T(\psi\cdot\phi) = T\psi\cdot T\phi$ for all $\psi,\phi \in L^\infty$, if and only if $T = T_{eg}$ for some measure-preserving function $g$.  
  By the linearity and continuity of $T= T_C$, it then suffices to show that $T(\one_B\cdot\one_E) = T\one_B\cdot T\one_E$ for all $B,E \in\mathscr{B}$. In fact, if $T\one_B = \one_{B'}$ and $T\one_E = \one_{E'}$ then $T\one_{B\cap E} = \one_{B'\cap E'}$ (by Proposition \ref{prop:sigma}\eqref{sigmacap}) and the desired equality follows.
\end{proof}

\begin{lem} \label{lem:TmuAxB}
    Let $C$ be a copula with associated Markov operator $T_C$ and doubly stochastic measure $\mu_C$ and $R,S \subseteq \mathscr{B}$.  $T_C\one_S = \one_R$ if and only if $\mu_C(R\times S) = \lambda(R) = \lambda(S)$.
\end{lem}
\begin{proof}
  Note the fact that $\mu_C(R\times S) = \int_R T_C\one_S \,d\lambda$. If $T_C\one_S = \one_R$ then $\mu_C(R\times S) = \lambda(R)$ and, by the property of $T_C$, $\lambda(R) = \int \one_R\,d\lambda = \int T_C\one_S \,d\lambda = \int \one_S \,d\lambda = \lambda(S)$.  Conversely, if $\mu_C(R\times S) = \lambda(R) = \lambda(S)$ then $\int_R T_C\one_S\,d\lambda = \mu_C(R\times S) = \lambda(S) = \int \one_S \,d\lambda = \int T_C\one_S \,d\lambda$ and $\int_R T_C\one_S \,d\lambda = \mu_C(R\times S) = \lambda(R) = \int_R \one \,d\lambda.$  Since $T_C\one_S(t) \in [0,1]$, $T_C\one_S(t) = 0$ for all $t \notin R$ and $T_C\one_S (t) = 1$ for all $t\in R$, i.e.~$T_C\one_S = \one_R$.
\end{proof}

\begin{prop} \label{prop:Markovfg}
  Let $T$ be a Markov operator and $f$ and $g$ be measure-preserving functions on $[0,1]$.  
  Then the following are equivalent.
  \begin{enumerate}
    \item \label{prop:Markovfg:2} $T\one_{g^{-1}(B)} = \one_{f^{-1}(B)}$ for all $B\in\mathscr{B}$.
    \item \label{prop:Markovfg:3} $T(\theta\circ g) = \theta\circ f$ for all Borel functions $\theta \in L^1$.
    \item \label{prop:Markovfg:4} $T\psi  = T_{ef}\circ T_{ge} \psi$ for all $g^{-1}(\mathscr{B})$-measurable functions $\psi$. 
  \end{enumerate}
\end{prop}
\begin{proof}
%
%
\ref{prop:Markovfg:2} $\Rightarrow$ \ref{prop:Markovfg:3}: 
For every Borel set $B\subseteq I$, $T(\one_B\circ g) = T \one_{g^{-1}(B)} = \one_{f^{-1}(B)} = \one_B\circ f$.  
So, 
$T(\psi\circ g) = \psi\circ f$ for all simple Borel functions $\psi$.  By the standard measure-theoretic argument, $T (\theta\circ g) = \theta\circ f$ for every Borel function $\theta$. 

 \ref{prop:Markovfg:3} $\Rightarrow$ \ref{prop:Markovfg:4}:
 Using Theorem \ref{thm:2.11DO2010}, \ref{prop:Markovfg:3}.~implies that $T_{fe}\circ T \circ T_{eg} \theta = T_{fe}\paren{T\paren{\theta\circ g}} = T_{fe}\paren{\theta\circ f} = \theta$ for every Borel function $\theta$, and hence $T_{fe}\circ T \circ T_{eg} = I = T_M$, that is $C_{fe} * C * C_{eg} = M$.  
This means that $C_{fe}*C$ is the unique left inverse of $C_{eg}$, or $C_{fe}*C = C_{ge}$.  Left multiplying by $C_{ef}$ yields $C_{ef}*C_{fe}*C = C_{ef}* C_{ge}$, i.e.~$T_{ef} \circ T_{fe}\circ T = T_{ef}\circ T_{ge}.$  On the other hand, by 
Theorem \ref{thm:2.11DO2010}, 
for any Borel function $\theta$, $T_{ef} \circ T_{fe} (\theta\circ f) =  (\theta\circ f)$ which gives $T_{ef} \circ T_{fe} \circ T(\theta\circ g) =  T(\theta\circ g).$ 
    Since $\set{\theta\circ g\colon \theta \text{ is a Borel function}}$ coincides with the class of $g^{-1}(\mathscr{B})$-measurable functions, $T_{ef} \circ T_{fe} \circ T = T$ on the class of $L^1$-functions which are $g^{-1}(\mathscr{B})$-measurable.  Therefore, $T\psi = (T_{ef}\circ T_{ge})\psi$ for all ($g^{-1}(\mathscr{B}) = \sigma_C$)-measurable $\psi \in L^1$. 

 \ref{prop:Markovfg:4} $\Rightarrow$ \ref{prop:Markovfg:2}:  This is clear from taking $\psi = \one_{g^{-1}(B)}$.
\end{proof}

The equivalence relation $\approx$ on $\mathscr{B}$ is defined as follows: $E\approx F$ if and only if the symmetric difference $E\triangle F$ has Lebesgue measure zero.  Of course, $\approx$ is still an equivalence relation on any $\mathscr{S} \subseteq \mathscr{B}$.  
The equivalence class of $S$ in $\mathscr{S}$ is denoted by $[S]_{\mathscr{S}}$ or just $[S]$ if no confusion can arise.  The collection of equivalence classes in $\mathscr{S}$ is denoted by $[\mathscr{S}]$.  $[\mathscr{S}]$ is in fact a measure algebra induced by the Lebesgue measure $\lambda$. That is, $[\mathscr{S}]$ is a Boolean $\sigma$-algebra with respect to the operations $[S]\vee [R] = [S\cup R]$ and $[S]\wedge [R] = [S\cap R]$ together with $\lambda\colon [\mathscr{S}] \to [0,1]$ defined overloadedly by $\lambda([S]) = \lambda(S)$ and satisfying $\lambda\paren{\bigvee_{i=1}^\infty [A_i]} = \sum_{i=1}^\infty \lambda([A_i])$ if $[A_i]\wedge[A_j]=[\emptyset]$ for $i\neq j$.  See \cite[p.398]{Roy1988-Real}.

$T_C$ induces a well-defined equivalence class function $\Upsilon_C\colon [\sigma_C] \to [\sigma^*_C]$ mapping $[S]$ to $[R]$ if and only if $T_C\mathds{1}_S = \mathds{1}_R$. \draftnote{Check!}
It follows from the defining property ii of $T$ that $\lambda(R) = \lambda(S)$ and hence $\Upsilon_C$ is measure-preserving.
The following lemma summarizes results from Theorem 3.2 and the proof of Theorem 3.3 in \cite{DarOls2010-Characterization}. 
\begin{lem}[\cite{DarOls2010-Characterization}] \label{lem:equivclassmap}
  Let $\mathscr{S}\subseteq\mathscr{B}$ be a non-atomic $\sigma$-algebra. Then 
  \begin{enumerate}
    \item there exists a surjective isomorphism $\Psi\colon \brac{\mathscr{S}} \to \brac{\mathscr{B}}$ which means $\Psi([S]^c)=\Psi([S])^c$, $\Psi\paren{[S_1]\vee[S_2]} = \Psi\paren{[S_1]}\vee\Psi\paren{[S_2]}$, $\lambda([S]) = \lambda(\Psi([S]))$ and $\Psi$ is onto (it follows that $\Psi$ is one-to-one and an isometry with respect to the metric $\rho([S],[R]) = \lambda([S]\triangle[R])$ and preserves countable unions and intersections); and 
    \item for any surjective isomorphism $\Psi\colon \brac{\mathscr{S}} \to \brac{\mathscr{B}}$, there exists a unique (a.e.) measure-preserving Borel function $h\colon [0,1] \to [0,1]$ such that $h^{-1}(\mathscr{B}) \subseteq \mathscr{S}$ (in fact, they are essentially equivalent,) and that $h^{-1}(B) \in \Psi^{-1}([B])$ for all $B\in\mathscr{B}.$  
  \end{enumerate}
\end{lem}

\begin{proof}
Enumerate $\Q\cap[0,1] = \set{r_n}_{n\in\N}$ and set $I_n = [0,r_n]$.  For each $n\in\N$, choose $S_n$ in the equivalence class $\Psi^{-1}([I_n])$ so that $S_k=[0,1]$ if $r_k=1$ and $S_k\subset S_l$ whenever $r_k < r_l$.
Define a Borel measure-preserving function $h(x) = \inf\set{r_k\colon x\in S_k}$.
For $S\in\mathscr{S}$, choose $B_0\in\Psi([S])$ and set $S_0 = h^{-1}(B_0)$.  Since $\Psi$ is 1-1, $[S] = \Psi^{-1}([B_0])$.  We prove that $h^{-1}(B_0) \in \Psi^{-1}([B_0])$ by considering $\mathscr{M} = \set{B\in\mathscr{B}\colon h^{-1}(B) \in \Psi^{-1}([B])}.$
Since $\mathscr{M}$ is a monotone class containing all $[0,r_k]$, it contains all Borel sets and we have the claim.
To prove uniqueness, 
if $k\colon[0,1]\to[0,1]$ is a measure-preserving Borel function such that $k^{-1}(B) \in \Psi^{-1}([B])$ for all $B\in\mathscr{B}$, then $k^{-1}([0,r_n]) \in \Psi^{-1}([[0,r_n]])$.  So $\lambda\paren{h^{-1}([0,r_n])\triangle k^{-1}([0,r_n])}=0$ for all $n$.  
By Proposition 10 
in \cite[Chapter 11 (p.~261)]{Roy1988-Real}, $h=k$ a.e.
\end{proof}

\begin{thm} \label{thm:Char-non-atomic}
%

  Let $C$ be a copula with associated Markov operator $T_C$ and doubly stochastic measure $\mu_C$. For measure-preserving functions $f$ and $g$ on $[0,1]$, the following are equivalent.
\begin{enumerate}
  \item \label{MainThm:graphf=g} $\mu_C\paren{\graph\set{f=g}} = 1$, where $\graph\set{f=g} = \set{(x,y)\colon f(x) = g(y)}$.
  \item \label{MainThm:Markovfg} $T_C = T_{ef}\circ T_{ge} $ on the class of $g^{-1}(\mathscr{B})$-measurable functions. 
  \item \label{MainThm:nonatomic} $C$ is non-atomic with $\sigma_C \supseteq g^{-1}(\mathscr{B})$ and $\sigma_C^*\supseteq f^{-1}(\mathscr{B})$.
\end{enumerate}
\end{thm}
\begin{proof} 
\ref{MainThm:graphf=g} $\Rightarrow$ \ref{MainThm:Markovfg}: 
By Proposition \ref{prop:Markovfg}, Lemma \ref{lem:TmuAxB} and the measure-preserving property of $f$ and $g$, it suffices to show that $\mu_C(f^{-1}(B) \times g^{-1}(B)) = \lambda(f^{-1}(B))$ for all $B\in\mathscr{B}$.  This follows from 
$\mu_C\paren{f^{-1}(B)\times g^{-1}(B)} \leq \mu_C\paren{f^{-1}(B)\times I} = \lambda\paren{f^{-1}(B)}$  
and 
\begin{align*}
  \mu_C\paren{f^{-1}(B)\times g^{-1}(B)} 
  &\geq \mu_C\paren{\set{(x,y)\colon f(x) = g(y) \in B}} \\
  &= \mu_C\paren{\set{(x,y)\colon f(x) = g(y)}\cap\set{(x,y)\colon f(x) \in B}} \\
  &=  \lambda\paren{f^{-1}(B)}. 
\end{align*}

\ref{MainThm:Markovfg} $\Rightarrow$ \ref{MainThm:graphf=g}: For each $n\in\N$ and $i=1,2,\dots,2^n$, put $I_{i,n} = \brac{\frac{i-1}{2^n}, \frac{i}{2^n}}$ and $B_n = \bigcup_i f^{-1}(I_{i,n})\times g^{-1}(I_{i,n})$.  Then $B_1 \supseteq B_2 \supseteq \cdots \supseteq B_n \supseteq \cdots$ and 
\begin{align*}
\mu_C\paren{B_n} 
&= \sum_{i=1}^{2^n} \mu_C\paren{f^{-1}(I_{i,n})\times g^{-1}(I_{i,n})} = \sum_{i=1}^{2^n} \lambda\paren{f^{-1}(I_{i,n})}
=  \sum_{i=1}^{2^n} \lambda\paren{I_{i,n}} = 1,
\end{align*} 
where we have used Lemma \ref{lem:TmuAxB} in the second equality and the measure-preserving property of $f$ in the third.  Set $B = \bigcap_{n=1}^\infty B_{n}$.  We have $\displaystyle \mu_C(B) = \lim_{n\to\infty} \mu_C(B_n) = 1$.
It then suffices to show that  $B = \set{(x,y)\colon f(x) = g(y)}$.  
First, if $f(x) \neq g(y)$ then there exists $n\in\N$ such that $\frac{1}{2^n} < |f(x) - g(y)|$ and hence $(x,y) \notin B_n$.  Conversely, it is clear that  $\set{(x,y)\colon f(x) = g(y)} = \bigcup_i \set{(x,y)\colon f(x) = g(y) \in I_{i,n}} \subseteq B_n$ for all $n$. 

\ref{MainThm:Markovfg} $\Rightarrow$ \ref{MainThm:nonatomic}:  By Proposition \ref{prop:Markovfg}, $T_C\one_{g^{-1}(B)} = \one_{f^{-1}(B)}$ for all $B\in\mathscr{B}$, which implies that $g^{-1}(\mathscr{B}) \subseteq \sigma_C$ and  $f^{-1}(\mathscr{B}) \subseteq \sigma^*_C$.  But both $g^{-1}(\mathscr{B})$ and $f^{-1}(\mathscr{B})$ are non-atomic because $g$ and $f$ are measure-preserving.  Therefore the finer $\sigma$-algebras $\sigma_C$ and $\sigma^*_C$ are non-atomic as well.

\ref{MainThm:nonatomic} $\Rightarrow$ \ref{MainThm:Markovfg}: Our proof is in three steps.
i) By applying Lemma \ref{lem:equivclassmap}(1) to $\sigma_C$ and $\sigma_C^*$, 
it follows that $\Upsilon_C\colon \brac{\sigma_C} \to \brac{\sigma_C^*}$ defined earlier is one-to-one, onto, measure-preserving and preserves order, complementation and the lattice operation on equivalence classes corresponding to countable unions of monotonic sequence of sets, i.e.~$\Upsilon_C$ is a surjective isomorphism.
ii) By Lemma \ref{lem:equivclassmap}(2), there exists  $\Xi_C\colon \brac{\sigma^*_C} \to \brac{\mathscr{B}}$  
    and a unique (a.e.) Borel function $f\colon I\to I$ such that $f^{-1}(B) \in \Xi_C^{-1}([B])$ for all $B\in\mathscr{B}.$
    So the composition $\Xi_C\circ \Upsilon_C\colon \brac{\sigma_C} \to \brac{\mathscr{B}}$ satisfies the same properties and induces a unique Borel function $g\colon I\to I$ such that 
    \begin{equation} \label{eq:g}
      g^{-1}(B) \in (\Xi_C\circ\Upsilon_C)^{-1}([B]) \qquad\text{ for all $B\in\mathscr{B}$.}
    \end{equation} 
%
iii) 
From \eqref{eq:g}, $\Upsilon_C\paren{[g^{-1}(B)]} = {\Xi_C^{-1}([B])}$ and,  by the definition of $\Upsilon_C$, $T_C \one_{g^{-1}(B)} = \one_{f^{-1}(B)}$ for all $B\in\mathscr{B}$. And the proof is done by Proposition \ref{prop:Markovfg}.
\end{proof}

In particular, for measure-preserving functions $f,g$ on $I$, the support of $C_{ef}*C_{ge}$ ``is'' the graph of $f(x) = g(y)$, that is $\mu_{C_{ef}*C_{ge}}(\graph\set{f=g}) = 1$, and its mass is distributed uniformly in the sense that $T_{C_{ef}*C_{ge}} \one_B = T_{ef}\paren{ T_{ge} \one_B }$ is $(\sigma_C^*= f^{-1}(\mathscr{B}))$-measurable for all $B\in\mathscr{B}$. 
%

\begin{exam}
For a fixed $\alpha \in (0,1)$, consider $C_\alpha = \alpha M + (1-\alpha) W$ with Markov operator $T_\alpha$.  
It is readily verified that $\sigma_{C_\alpha}=\sigma^*_{C_\alpha}= \set{S\in\mathscr{B}\colon S = 1-S}$ and $T_\alpha \one_S = \one_S$ for all $S \in \sigma_{C_\alpha}$.  But only $T_{\frac{1}{2}}$ has the property that $T_{\frac{1}{2}}\one_B$ is  $\sigma_{C_{\frac{1}{2}}}$-measurable for all $B\in\mathscr{B}$.  Note also that $C_{\frac{1}{2}} = L_{\frac{1}{2}}*L^t_{\frac{1}{2}}$.
\end{exam}

\section{Copulas with fractal support}
Let us recall the construction of copulas with self-similar supports in Fredricks et al.~
\cite{FreNelRod2005-Copulas} put in the context of patched copulas \cite{YanJinJia2011-Approximation,ChaSanSum2016-Patched}.
\begin{defn}
A \emph{transformation matrix} is a matrix $A$ with nonnegative entries, for which the sum of all entries is $1$ and 
every row and column has at least one non-zero entry.
\end{defn}
Given a transformation matrix $A = [a_{ij}]_{k\times \ell}$ where the first index ($i$) is the column number from left ($i=1$) to right  ($i=k$) and the second index ($j$) is the row number from bottom ($j=1$) to top ($j=\ell$), the matrix multiplication of $A$ and $B=[b_{jm}]_{\ell\times n}$ is defined by $[a_{ij}] [b_{jm}] = \brac{\sum_{j=1}^\ell a_{ij}b_{jm}}$.  
$A^t = [a_{ji}]_{ \ell\times k}$.
This unconventional entry arrangement syncs well with the product of copulas defined \cite{DarNguOls1992-Copulas} as 
\[C*D (u,v) = \int_0^1 \partial_2C(u,t) \partial_1D(t,v) \,dt.\] 
Let $p_i$ denote the sum of the entries in the first $i$ columns of $A$ and let $q_j$ denote the sum of the entries in the first $j$ rows of $A$ where $p_0 = 0$ and $q_0=0$. 
The partitions $\set{p_i}_{i=0}^k$ and $\set{q_j}_{j=0}^\ell$ of $[0,1]$ yield a rectangular partition of $[0,1]^2$ consisting of $k\cdot\ell$ rectangles with vertices $(p_i,q_j)$.
Given a copula $C$, 
we construct a new copula denoted by $[A](C)$ by placing a scaled copy of $C$ in each of the $k\cdot\ell$ rectangles weighted according to the mass given by the corresponding entry in $A$. It is defined for $(u,v) \in R_{ij} = [p_{i-1},p_i]\times [q_{j-1},q_j]$ by 
\begin{align*}
  [A](C)(u,v) 
  &= \sum_{i'<i, j'<j}a_{i'j'} + \frac{u-p_{i-1}}{p_i-p_{i-1}}\sum_{j'<j} a_{ij'} 
  + \frac{v-q_{j-1}}{q_j-q_{j-1}}\sum_{i'<i} a_{i'j}\\
  &\quad + a_{ij} C\paren{\frac{u-p_{i-1}}{p_i-p_{i-1}},\frac{v-q_{j-1}}{q_j-q_{j-1}}}.
\end{align*}
Here, empty sums are zero by convention.   See  \cite{FreNelRod2005-Copulas} for more details.  
It will be more convenient to view $[A](C)$ as the so-called \emph{patched copula} \cite{YanJinJia2011-Approximation,ChaSanSum2016-Patched} defined as  
\[[A](C)(u,v) = \sum_{i=1}^k \sum_{j=1}^\ell a_{ij} C\paren{F_i(u),G_j(v)}\]
where \label{Fi_Gj} $F_i(u) = \min\paren{1,\dfrac{u-p_{i-1}}{p_i - p_{i-1}}}\one_{(p_{i-1}, \infty)}(u)$ is the uniform distribution on $[p_{i-1},p_i]$ and $G_j(v) = \min\paren{1,\dfrac{v-q_{j-1}}{q_j - q_{j-1}}}\one_{(q_{j-1}, \infty)}(v)$ is the uniform distribution on $[q_{j-1},q_j]$.  
Denote $\triangle p_i = p_i - p_{i-1}$ and  $\triangle q_j = q_j - q_{j-1}$. 

 
We then investigate how this rectangular patching of a copula $C$ according to the transformation matrix $A$ affects $C$ in terms of their Markov operators.
\begin{align}
  T_{[A](C)}f (x) 
  &= \sum_{i=1}^k \sum_{j=1}^\ell \frac{a_{ij}}{\triangle q_j} \frac{d}{dx} \int_{q_{j-1}}^{q_j} \partial_2 C\paren{F_i(x),G_j(t)} f(t)\,dt \notag\\
  &= \sum_{i=1}^k \sum_{j=1}^\ell a_{ij} \frac{d}{dx} \int_{0}^{1} \partial_2 C\paren{F_i(x),s} f\circ G_j^{-1}(s)\,ds \notag\\
  &= \sum_{i=1}^k \sum_{j=1}^\ell \frac{a_{ij}}{\triangle p_i} T_C(f\circ G_j^{-1})(F_i(x)). \label{eq:T_AC}
\end{align}
Consequently, $\displaystyle T_{[A](C)}(f)(F_i^{-1}(x)) = \sum_{j=1}^\ell \frac{a_{ij}}{\triangle p_i} T_C(f\circ G_j^{-1})(x)$, \draftnote{bdry pts of partition?} which can be written in matrix form as 
\[ \Big[ T_{[A](C)}(f)(F_i^{-1}(x)) \Big]_{k\times 1} = \Big[\frac{a_{ij}}{\triangle p_i}\Big]_{k\times\ell} \Big[T_C(f\circ G_j^{-1})(x)\Big]_{\ell  \times 1}.\]  
See \cite{Tru2012-Idempotent, TruSan2012-Idempotent} for essentially the same identity in terms of Markov kernels.

Define inductively $[A]^n(C) = [A]\paren{[A]^{n-1}(C)}$ for $n\geq 2$.   
Fredricks et al.~\cite{FreNelRod2005-Copulas} showed that for any transformation matrix $A \neq [1]$ and copula $C$, $[A]^n(C)$ converges (pointwise and hence uniformly) to a unique copula $C_A$, as $n\to\infty$.  Moreover, $C_A$ is the fixed point of $[A]$, i.e.~$[A](C_A) = C_A$.
However, since uniform convergence will not suffice for our purposes, we shall investigate the convergence of $[A]^k(D)$ with respect to a stronger norm with respect to which the $*$-product is jointly continuous. We choose the modified Sobolev norm defined as $\norm{C}^2_S = \norm{C}_1^2 + \norm{C}_2^2$ where $\norm{C}_i^2 = \int_0^1 \int_0^1 \abs{\partial_i C(u,v)}^2 \,du\,dv$.  See \cite{DarOls1995-Norms}.
\begin{prop}
  Let $A$ be a transformation matrix whose dimension is at least $2\times 2$ and let $C$ and $D$ be copulas. Then 
  \[\norm{[A](C) - [A](D)}_S \leq r\norm{C - D}_S\]
  where $\displaystyle r^2 = \max\paren{\sum_{j=1}^\ell \sum_{i=1}^k a_{ij}^2\frac{q_j - q_{j-1}}{p_i - p_{i-1}},\sum_{j=1}^\ell \sum_{i=1}^k a_{ij}^2\frac{p_i - p_{i-1}}{q_j - q_{j-1}}} < 1$ and $\norm{\cdot}_S$ is the modified Sobolev norm .
\end{prop}
\begin{proof}
Denote $A = [a_{ij}]_{k\times \ell}$ with $k,\ell \geq 2$ and let $\set{p_i}_{i=0}^k$, $\set{q_j}_{j=0}^\ell$ be the induced partitions of $[0,1]$ on the $x$-axis and $y$-axis, respectively. 
  For $u \in (p_{i-1},p_i)$ and $v\in (q_{j-1},q_j)$, 
  $\partial_1 [A](C)(u,v) = a_{ij} F_i'(u) \partial_1 C \paren{F_i(u), G_j(v)} + \sum_{k=1}^{j-1} a_{ik} F_i'(u)$ 
  and so 
  \[\abs{\partial_1 [A](C)(u,v) - \partial_1 [A](D)(u,v)} = \frac{a_{ij}}{p_i - p_{i-1}} \abs{\partial_1 C \paren{F_i(u), G_j(v)} - \partial_1 D \paren{F_i(u), G_j(v)}}.\]  
  Hence, $\norm{[A](C) - [A](D)}_1^2$ is equal to 
  \begin{align*}
    & \sum_{j=1}^\ell \sum_{i=1}^k \frac{a_{ij}^2}{(p_i - p_{i-1})^2} \int_{q_{j-1}}^{q_j} \int_{p_{i-1}}^{p_i}  \abs{\partial_1 C \paren{F_i(u), G_j(v)} - \partial_1 D \paren{F_i(u), G_j(v)}}^2 \,du\,dv\\
    &= \paren{\sum_{j=1}^\ell \sum_{i=1}^k a_{ij}^2\frac{q_j - q_{j-1}}{p_i - p_{i-1}}} \norm{C-D}_1^2.
  \end{align*}
A similar proof yields $\displaystyle \norm{[A](C) - [A](D)}_2^2 = \paren{\sum_{j=1}^\ell \sum_{i=1}^k a_{ij}^2\frac{p_i - p_{i-1}}{q_j - q_{j-1}}} \norm{C-D}_2^2$.  

Now, since $a_{ij} \leq p_i - p_{i-1}$ and $\sum_{i} a_{ij} = q_j - q_{j-1}$, 
\[\sum_{j=1}^\ell \sum_{i=1}^k a_{ij}^2\frac{q_j - q_{j-1}}{p_i - p_{i-1}} \leq \sum_{j=1}^\ell \sum_{i=1}^k a_{ij}\paren{q_j - q_{j-1}} = \sum_{j=1}^\ell (q_j - q_{j-1})^2 < \sum_{j=1}^\ell (q_j - q_{j-1}) = 1,\]
where we have used the assumption that $\ell \ge 2$ in the last inequality.  In fact, the same assumption implies that $a_{ij} < p_i - p_{i-1}$ for some $j$.  The same line of proof using $k\ge 2$ shows $\sum_{j=1}^\ell \sum_{i=1}^k a_{ij}^2\frac{p_i - p_{i-1}}{q_j - q_{j-1}} < 1$ and the desired result follows.
\end{proof}

We now have that the mapping 
$[A]$ is a contraction on the class of copulas $\mathcal{C}_2$ with respect to the modified Sobolev norm.  Using the fact \cite{DarOls1995-Norms} that $\mathcal{C}_2$ is complete with respect to the modified Sobolev norm, it follows from the Contraction-Mapping Theorem that: 
\begin{thm} \label{thm:SblConvFrac}
  For each transformation matrix $A$ of dimension at least $2\times 2$, 
  and for any initial copula $D$, the sequence $\set{[A]^r(D)}_r$ 
  converges to the copula $C_A$ in the modified Sobolev norm.  
\end{thm}


\begin{defn}
Let $A = [a_{ij}]_{k\times \ell}$ be a transformation matrix, 
$\emptyset\neq I \subseteq \set{1,2,\dots,k}$ and $\emptyset\neq J \subseteq \set{1,2,\dots,\ell}$. 
The pair $(I,J)$ is called an \emph{invariant pair} of $A$ if $a_{ij} = 0$ for all $(i,j) \in (I\times J^c)\cup(I^c\times J)$ and $a_{ij}>0$ for some $(i,j) \in I\times J$.
Two invariant pairs $(I_1,J_1)$ and $(I_2,J_2)$ of $A$ are called \emph{disjoint} if $I_1\cap I_2 = \emptyset$ and $J_1\cap J_2 = \emptyset$.  It follows that if they are not disjoint then both $I_1\cap I_2$ and $J_1\cap J_2$ are not empty.

We say that $A$ is \emph{disjointly decomposable} 
 if $A$ has a finite number of pairwise disjoint invariant pairs $(I_1,J_1), (I_2,J_2), \dots, (I_N,J_N)$ 
 such that $\bigcup_{n=1}^N I_n = \set{1,2,\dots,k}$ and $\bigcup_{n=1}^N J_n = \set{1,2,\dots,\ell}$.  
For each $n = 1,2,\dots,N$, let us denote by $A_n$ the $k\times\ell$ matrix whose $(i,j)$th entry is $a_{ij}$ if $i\in I_n$ and $j\in J_n$ and is equal to zero otherwise.  
Observe that $\displaystyle A = \sum_{n=1}^N A_n$ and in particular, every non-zero entry in $A$ appears in exactly one $A_n$. 
We also say that $A$ is \emph{disjointly decomposable} as the sum $\sum_{n=1}^N A_n$ or \emph{disjointly decomposable} by $N$ invariant pairs.

Such a disjoint decomposition of $A$ gives rise to two partitions of $[0,1]$: $\set{Q_n}_{n=1}^N$ and $\set{P_n}_{n=1}^N$ defined by
$Q_n = \bigcup_{j\in J_n} (q_{j-1},q_j)$ and $P_n = \bigcup_{i\in I_n} (p_{i-1},p_i)$. Then $\lambda(Q_n)>0$ and $\lambda(P_n)>0$ for all $n=1,\dots,N$.
Each pair $(I_n,J_n)$ induces set functions $\mathscr{G}_n$ and $\mathscr{F}_n$ mapping $B \in \mathscr{B}$ to 
\[\mathscr{G}_n(B) \equiv \bigcup_{j\in J_n} G_j^{-1}(B) \subseteq Q_n \quad \text{and}\quad 
\mathscr{F}_n(B) \equiv \bigcup_{i\in I_n} F_i^{-1}(B) \subseteq P_n.\]  Note that $G_j^{-1}((0,1)) = (q_{j-1},q_j)$ and $F_i^{-1}((0,1)) = (p_{i-1},p_i)$.
\end{defn}

\begin{remk}
 If $A$ has $N$ invariant pairs $(I_1,J_1), (I_2,J_2), \dots, (I_N,J_N)$, not necessarily disjoint, such that $\bigcup_{n=1}^N I_n = \set{1,2,\dots,k}$ and $\bigcup_{n=1}^N J_n = \set{1,2,\dots,\ell}$ then $A$ is disjointly decomposable by $N'\geq N$ invariant pairs.  In fact, without loss of generality, if $(I_1,J_1)$ and $(I_2,J_2)$ are not disjoint then $I' \equiv I_1\cap I_2 \neq \emptyset$ and $J' \equiv J_1\cap J_2 \neq \emptyset$ and it can be shown that $(I_1\setminus I_2, J_1\setminus J_2), (I_2\setminus I_1, J_2\setminus J_1), (I',J')$ are pairwise disjoint invariant pairs. \draftnote{Proof?}  Replacing $(I_1,J_1), (I_2,J_2)$ with these three pairs gives us a finer list of invariant pairs.  Then repeat the process until all invariant pairs are pairwise disjoint.  
\end{remk}

\begin{lem} \label{lem:A(C)}
  If $(I,J)$ is an invariant pair of a transformation matrix $A$ and $S,R\in\mathscr{B}$ are such that $T_C\one_S = \one_R$, then $T_{[A](C)}\one_{\bigcup_{j\in J}G_j^{-1}(S)} = \one_{\bigcup_{i\in I}F_i^{-1}(R)}.$
\end{lem}
\begin{proof}
For $j\in J$, we have $\one_{\bigcup_{j'\in J}G_{j'}^{-1}(S)}\circ G_j^{-1}(x) = \one_S(x)$; otherwise, $\one_{\bigcup_{j'\in J}G_{j'}^{-1}(S)}\circ G_j^{-1}(x) = 0$.  So, by \eqref{eq:T_AC},
  \begin{align*}
    T_{[A](C)}\one_{\bigcup_{j\in J}G_j^{-1}(S)} 
 &= \sum_{i=1}^k \sum_{j=1}^\ell \frac{a_{ij}}{\triangle p_i} T_C\left(\one_{\bigcup_{j'\in J}G_{j'}^{-1}(S)}\circ G_j^{-1}\right)\left(F_i(x)\right)\\
 &= \sum_{i=1}^k \sum_{j\in J} \frac{a_{ij}}{\triangle p_i} T_C(\one_{S})(F_i(x))\\
 &= \sum_{i\in I}  \frac{\sum_{j\in J} a_{ij}}{\triangle p_i} T_C(\one_{S})(F_i(x)),
  \end{align*}
where we have used the assumption that $a_{ij} = 0$ if $i\notin I$ and $j\in J$.
Since for each $i\in I$, $a_{ij}=0$ for all $j\notin J$, the sum $\sum_{j\in J} a_{ij}$ is equal to $\triangle p_i$ and 
\[T_{[A](C)}\one_{\bigcup_{j\in J}G_j^{-1}(S)} = \sum_{i\in I} T_C(\one_{S})(F_i(x)) =  \sum_{i\in I} \one_R(F_i(x)) = \one_{\bigcup_{i\in I} F_i^{-1}(R)}(x).\qedhere\]
\end{proof}

%
 
\begin{lem} \label{lem:SnA(C)}
Let $C$ be a copula.  
Suppose a transformation matrix $A$ is disjointly decomposable by $N$ invariant pairs.
  If $T_{[A](C)}\one_S = \one_R$ 
  then for $n=1,\dots,N$, $T_{[A](C)}\one_{S_n} = \one_{R_n}$ where $S_n \equiv S\cap Q_n$ and $R_n\equiv R\cap P_n$. 
\end{lem}
\begin{proof}
  For each $n=1,\dots,N$, the positivity of $T_{[A](C)}$ implies that $T_{[A](C)}\one_{S_n} \leq T_{[A](C)}\one_{S} = \one_R$.  If it holds that $T_{[A](C)}\one_{S_n} = 0$ for a.e.~$x$ not in $P_n$, then $T_{[A](C)}\one_{S_n} \leq \one_{R\cap P_n} = \one_{R_n}$.  Summing over all $n$ gives \[\one_R = T_{[A](C)}\one_S = \sum_n T_{[A](C)}\one_{S_n} \leq \sum_n \one_{R_n} = \one_R\] and hence $T_{[A](C)}\one_{S_n} = \one_{R_n}$ for all $n$.
We then prove the claim. If $x \notin P_n \equiv \bigcup_{i'\in I_n} (p_{i'-1},p_{i'})$, then $F_i(x) = 0$ or $1$ for all $i\in I_n$ and 
\begin{align*}
T_{[A](C)}\one_{S_n}(x) 
&= \sum_{i\in I_n} \sum_{j\in J_n} \frac{a_{ij}}{\triangle p_i} T_C\paren{\sum_{j'\in J_n}\one_{S\cap(q_{j'-1},q_{j'})}\circ G_j^{-1}}(F_i(x))\\
&=  \sum_{i\in I_n} \sum_{j\in J_n} \frac{a_{ij}}{\triangle p_i} T_C\paren{\one_{S\cap(q_{j-1},q_{j})}\circ G_j^{-1}}(F_i(x)) = 0.
\end{align*}
We use the convention that $0$ and $1$ are not in the support of all functions.
\end{proof}

\begin{thm} \label{thm:C_A=nonatomic}
  Suppose a transformation matrix $A$ is disjointly decomposable by $N\geq 2$ invariant pairs. Then the invariant copula $C_A$ is non-atomic. 
\end{thm}
\begin{proof}
  Let $S,R \in\mathscr{B}$ be such that $T_{C_A}\one_S = \one_R$ and $\lambda(S) = \lambda(R) > 0$.  Since $C_A = [A](C_A)$, by Lemma \ref{lem:SnA(C)}, $T_{C_A}\one_{S_{n}} = \one_{R_{n}}$ for all $n=1,2,\dots,N$ where $S_{n} \equiv \mathscr{G}_n(S) \subseteq Q_{n}$ and $R_{n} \equiv \mathscr{F}_n(R) \subseteq P_{n}$ are such that $S = \bigcup_{n} S_{n}$ and $R=\bigcup_{n} R_{n}$.  If  
one of the $\lambda(S_{n})$'s is strictly between $0$ and $\lambda(S)$ then $S$ and $R$ are not atoms of $\sigma_{C_A}$ and $\sigma^*_{C_A}$, respectively.  Otherwise, there is an $n_1$ such that $\lambda(S_{n_1})=\lambda(S)$ and we repeat the process by applying Lemma \ref{lem:SnA(C)} to $T_{C_A}\one_{S_{n_1}} = \one_{R_{n_1}}$ and obtain $S_{n_1,n}\equiv \mathscr{G}_n(S_{n_1}) \subseteq \mathscr{G}_n(Q_{n_1})$ and $R_{n_1,n} \equiv \mathscr{F}_n(R_{n_1}) \subseteq \mathscr{F}_n(P_{n_1})$ are such that $S_{n_1} = \bigcup_{n} S_{n_1,n}$, $R_{n_1} = \bigcup_{n} R_{n_1,n}$ and $T_{C_A}\one_{S_{n_1,n}} = \one_{R_{n_1,n}}$ for all $n=1,2,\dots,N$.  If some $\lambda(S_{n_1,n})$ lies between $0$ and $\lambda(S)$ then we are done.  Otherwise, there must be an $n_2$ such that $\lambda(S_{n_1,n_2})=\lambda(S)$.  This process will certainly stop because $\lambda(\mathscr{G}_{n_k}( \cdots(\mathscr{G}_{n_2}(Q_{n_1})))) = \lambda(Q_{n_k})\cdots\lambda(Q_{n_2})\lambda(Q_{n_1}) \to 0$ as $k\to\infty$ and $\lambda(S) > 0$.
\end{proof}

\begin{exam} \label{exam:nonatomic}
Let $A_1 = K_0 + K_1$ and $A_2 = K_0 + K_2$ where 
\[ K_0 = \begin{bmatrix} 0 & 0 & 0\\ 0 & 1/3 & 0\\ 0 & 0 & 0 \end{bmatrix},\; 
   K_1 = \begin{bmatrix} 1/12 & 0 & 1/4\\ 0 & 0 & 0\\ 1/4 & 0 & 1/12 \end{bmatrix}\; \text{ and } \; 
   K_2 = \begin{bmatrix} 1/6 & 0 & 1/6\\ 0 & 0 & 0\\ 1/6 & 0 & 1/6 \end{bmatrix}.\]
Then both transformation matrices $A_1$ and $A_2$ are disjointly decomposable by $2$ invariant pairs $(\set{1,3},\set{1,3})$ and $(\set{2},\set{2})$.  By Theorem \ref{thm:C_A=nonatomic}, the invariant copulas $C_{A_1}$ and $C_{A_2}$ are non-atomic and hence, by Theorem \ref{thm:Char-non-atomic}, it is supported in the support of an implicit dependence copula.  In fact, they share the same support shown in Figure \ref{fig:1}.  Note also that both copulas are symmetric as the corresponding matrices are.  
%
We will see later that (only) $C_{A_2}$ can be factored as $L_1*L_2^t$ for some left invertible copulas $L_1,L_2$.
\begin{figure}
\includegraphics[width=\textwidth]{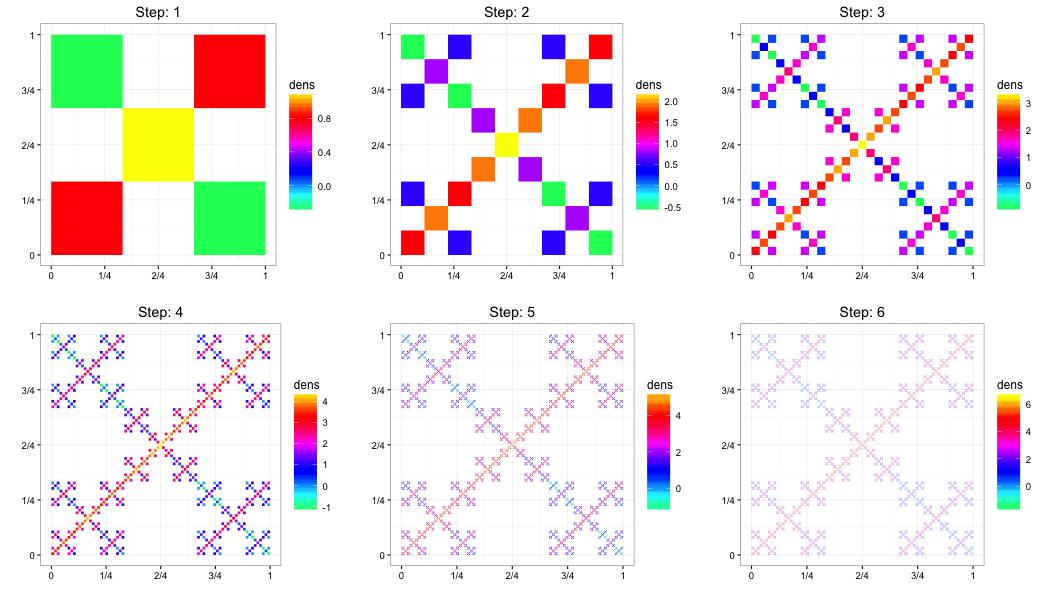} 
\caption{Supports of $[A_1]^r(\Pi)$, $r=1,2,\dots,5$}
\label{fig:1}
\end{figure}
\end{exam}

A transformation matrix $L=[\lambda_{in}]_{k\times N}$ is said to be a \emph{left complete dependence matrix} if there is exactly one nonzero entry in each column, i.e.~for each $i$, there exists a positive integer $n_i\leq N$ such that $\lambda_{in} = 0$ for $n\neq n_i$.  \emph{Right complete dependence matrices} are defined similarly. 
\begin{lem} \label{lem:L(C)}
  Let $L$ be a left complete dependence matrix with dimension at least $2\times 2$ and $C$ be a left invertible copula. Then $[L](C), [L]^2(C), \dots, [L]^r(C), \dots$ are left invertible copulas converging to a left invertible copula $C_L$. 
  The same statement also holds if all occurrences of ``left'' are replaced by ``right.''
\end{lem}
\begin{proof}
Since the general case can be proved by induction on $r$, it suffices to show that $[L](C)$ is left invertible.
By definition, the $x$-partition of $[L](C)$ is determined by the only non-zero entry in each column: $p_0=0$ and $p_i = \sum_{i'\leq i} \lambda_{i'n_{i'}}$, $i=1,\dots,k$.  
So $[L](C)(u,v) = \sum_{i=1}^k  \lambda_{in_i} C\paren{F_i(u),G_{n_i}(v)}$ and 
\begin{align*}
 [L](C)^t * [L](C)(u,v) 
&= \sum_{i'=1}^k \sum_{i=1}^k  \lambda_{in_i} \lambda_{i'n_{i'}}  \int_0^1 \frac{\partial}{\partial t} C\paren{F_{i'}(t),G_{n_{i'}}(u)} 
\frac{\partial}{\partial t} C\paren{F_i(t),G_{n_i}(v)}\,dt\\
&= \sum_{i=1}^k  \lambda_{in_i}^2  \int_0^1 \partial_1 C\paren{F_{i}(t),G_{n_{i}}(u)} 
\partial_1 C\paren{F_i(t),G_{n_i}(v)}F_i'(t)^2\,dt\\
&= \sum_{i=1}^k  \lambda_{in_i} C^t*C(G_{n_{i}}(u),G_{n_i}(v)) \\
&= \sum_{i=1}^k  \lambda_{in_i} G_{n_{i}}(\min(u,v)) = M(u,v). 
\end{align*}  \draftnote{More explanation?}
Therefore, 
$[L](C)$ is left invertible. In general, we have $[L]^r(C)^t * [L]^r(C) = M$ and so, by taking the limit as $r\to\infty$ with respect to the modified Sobolev norm (see Theorem \ref{thm:SblConvFrac}), $C_L^t*C_L = M$.
\end{proof}


\begin{thm} \label{thm:LRfactor}
Let $A$ be a $k\times \ell$ transformation matrix, with $k,\ell\geq 2$, which is disjointly decomposable by $N\geq 2$ invariant pairs.  
If all $A_1, A_2, \dots, A_N$ have rank one, then $C_A = L*R$ 
for some left invertible copula $L$ and right invertible copula $R$.
\end{thm}
\begin{proof}
If $A_n$ has rank one, then there exist a row matrix $L_n = \brac{\lambda_{in}}_{k\times 1}$ and a column matrix $R_n = \brac{\rho_{nj}}_{1\times\ell}$ such that $|L_n| = |A_n| = |R_n|$ and 
\begin{equation} \label{eq:A_n}
A_n = \frac{1}{|A_n|}L_n R_n,
\end{equation}
where $|A|$ denotes the sum of the absolute values of all entries in a matrix $A$.   
Stacking up $L_n$'s vertically and $R_n$'s horizontally, we obtain transformation matrices 
\[L = \brac{\lambda_{in}}_{k\times N} = \brac{\begin{matrix}
L_N\\
\vdots\\
L_1
\end{matrix}}\; \text{ and } \; R = [\rho_{nj}]_{N\times \ell} = \brac{\begin{matrix}
R_1 & 
 \cdots & R_N
\end{matrix}}.\]
Since $I_n$'s are disjoint, each column of $L$ has exactly one non-zero entry.  Similarly, each row of $R$ has exactly one non-zero entry.

\commentout{
\begin{center}
\psset{xunit=3cm,yunit=3}
\begin{pspicture*}(-.2,-.2)(1.2,1.2)
  \psaxes[ticks=none,labels=none]{->}(0,0)(-1,-1)(1,1)
\psline(.7,-.03)(.7,.03)
\psline(-.03,.8)(.03,.8)
\uput[-90](.7,0){$x$}
\uput[180](0,.8){$y$}
\psline[linestyle=dashed](.7,0)(.7,.8)
\psline[linestyle=dashed](0,.8)(.7,.8)
\uput[120](.35,.4){$r$}
\pswedge[
fillstyle=solid]{.7}{0}{48}
\uput[40](.05,0){$\theta$}
  \psline[linecolor=red, linewidth=1.5pt]{-}(0,0)(.7,.8)
\uput[45](.7,.8){$P$}  
\psline[linewidth=1.5pt]{->}(0,0)(1,0)
\end{pspicture*}
\end{center}
}

We then show that $[L](C_1)*[R](C_2) = [A](C_1*C_2)$ for any copulas $C_1$ and $C_2$.  For $m=1,2,\dots,N$, denote by $H_m$ the uniform distribution on $\brac{\sum_{n=1}^{m-1}|A_n|,\sum_{n=1}^{m}|A_n|}$.  So $H_m' = \frac{1}{|A_m|}$ on its support.
For $u,v \in [0,1]$, 
\begin{align}
  & [L](C_1)*[R](C_2)(u,v) \notag \\
  &= \int_0^1 \paren{\sum_{i=1}^k \sum_{n=1}^N \lambda_{in} \partial_2 C_1\paren{F_i(u), H_n(t)}} 
  \paren{ \sum_{j=1}^\ell \sum_{m=1}^N \rho_{mj} \partial_1 C_2\paren{H_m(t), G_j(v)}} dt \notag \\
  &= \sum_{i=1}^k \sum_{j=1}^\ell \paren{\sum_{n=1}^N \frac{\lambda_{in}\rho_{nj}}{|A_n|}} C_1*C_2\paren{F_i(u), G_j(v)} \label{eq:LR=A} \\
  &= [A](C_1*C_2)(u,v). \label{main-eq:LR=A}
\end{align}
It is left to verify that the sum over $n$ in \eqref{eq:LR=A} is equal to the $(i,j)$th-element in $A$.  
Since $\sum_{n=1}^N A_n$ is a disjoint decomposition of $A = [a_{ij}]_{k\times\ell}$, the equation \eqref{eq:A_n} implies that if $(i,j) \in I_m\times J_m$ for some (unique) $m$ then \[a_{ij} = \dfrac{\text{the $(i,j)$th element of $L_mR_m$}}{|A_m|} = \dfrac{\lambda_{im}\rho_{mj}}{|A_m|} = \sum_{n=1}^N \dfrac{\lambda_{in}\rho_{nj}}{|A_n|},\]  where the last equality follows from the fact that $\lambda_{in}\rho_{nj} = 0$ if $(i,j) \notin I_n\times J_n$.
It also clearly follows from this fact that if if $(i,j) \notin I_m\times J_m$ for all $m$ then $ \sum_{n=1}^N \dfrac{\lambda_{in}\rho_{nj}}{|A_n|} = 0 = a_{ij}$.

By \eqref{main-eq:LR=A}, a proof by induction on $r = 1,2,\dots$ yields $[L]^r(C_1)*[R]^r(C_2) = [A]\paren{[L]^{r-1}(C_1) * [R]^{r-1}(C_2)} = \cdots = [A]^r(C_1*C_2)$.  So $[L]^r(E)*[R]^r(E) = [A]^r(E)$ for all $r\geq 1$ if $E$ is an idempotent copula, e.g.~$E=M$ or $\Pi$.
By Lemma \ref{lem:L(C)}, $[L]^r(M)$ are left invertible copulas and $[R]^r(M)$ are right invertible copulas.    

By Theorem \ref{thm:SblConvFrac}, with respect to the modified Sobolev norm, $[L]^r(M)$, $[R]^r(M)$ and $[A]^r(M)$ converge respectively to a left invertible copula $C_L$, a right invertible copula $C_R$ and a copula $C_A$ satisfying $[L](C_L) = C_L$, $[R](C_R) = C_R$ and $[A](C_A) = C_A$.  As a consequence of the joint continuity of the $*$-product with respect to $\norm{\cdot}_S$, we obtain $C_L *  C_R = C_A$ as desired.
\end{proof}


\begin{exam} \label{exam:r}
 Fix $r\in (0,1/2)$ and consider the transformation matrix 
 \[\begin{bmatrix} r/2 & 0 & r/2\\ 0 & 1-2r & 0\\ r/2 & 0 & r/2 \end{bmatrix} = 
 \begin{bmatrix} 0 & 0 & 0\\ 0 & 1-2r & 0\\ 0 & 0 & 0 \end{bmatrix} + 
 \begin{bmatrix} r/2 & 0 & r/2\\ 0 & 0 & 0\\ r/2 & 0 & r/2 \end{bmatrix}.\]
Both matrices on the right hand side have rank one which implies by Theorem \ref{thm:LRfactor} that the invariant copula can be factored as the product of a left invertible copula and a right invertible copula.  When $r=1/3$, the transformation matrix is $A_2$ in Example \ref{exam:nonatomic}.  The copula factors, $C_{L_2}$ and $C_{R_2}$, of $C_{A_2}$ are shown (approximately) in Figure \ref{fig:2}, where $L_2 = \begin{bmatrix}  0 & 1/3 & 0\\ 1/3 & 0 & 1/3 \end{bmatrix} = R_2^t$.
\begin{figure}
\includegraphics[width=0.32\textwidth]{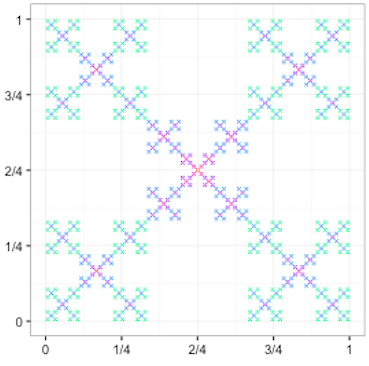} 
\includegraphics[width=0.32\textwidth]{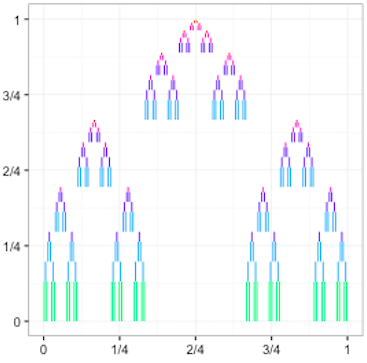} 
\includegraphics[width=0.32\textwidth]{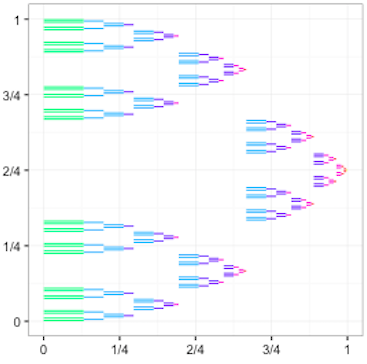} 
\caption{Supports of $[A_2]^5(\Pi)$, $[L_2]^5(\Pi)$ and $[R_2]^5(\Pi)$ in Example \ref{exam:r}}
\label{fig:2}
\end{figure}
\end{exam}

\begin{exam} \label{exam:nonsymLR}
Let us also consider the factorization of a non-symmetric invariant copula 
by first setting
\[K_3 = \begin{bmatrix} 3/28 & 0 & 27/140\\ 0 & 0 & 0\\ 1/7 & 0 & 9/35 \end{bmatrix}, \;
L_3 = \begin{bmatrix}  0 & 3/10 & 0\\ 1/4 & 0 & 9/20 \end{bmatrix}\;\text{ and } \;
R_3 = \begin{bmatrix}  3/10 & 0\\  0 & 3/10\\ 4/10 & 0\end{bmatrix}.\]
Then the non-symmetric transformation matrix $A_3 = K_0 + K_3$ is disjointly decomposable by $2$ invariant pairs $(\set{1,3},\set{1,3})$ and $(\set{2},\set{2})$.  Since $K_0$ and $K_3$ have rank one, it follows from Theorem \ref{thm:LRfactor} and its proof that $C_{A_3} = C_{L_3}*C_{R_3}$.  Approximations of their supports are illustrated in Figure \ref{fig:3}.
\begin{figure}
\includegraphics[width=0.32\textwidth]{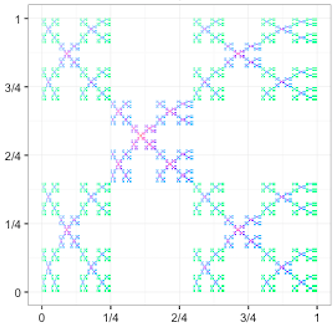} 
\includegraphics[width=0.32\textwidth]{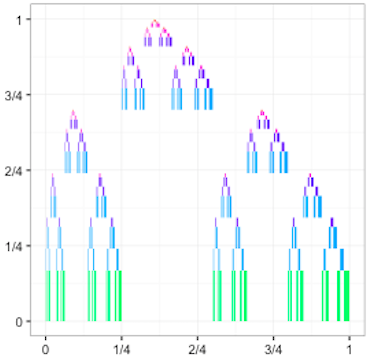} 
\includegraphics[width=0.32\textwidth]{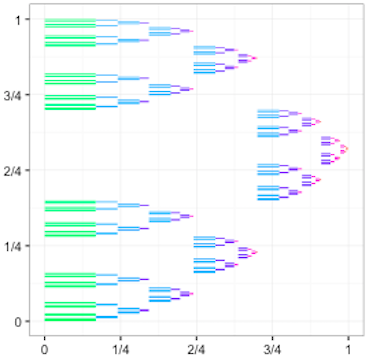} 
\caption{Supports of $[A_3]^5(\Pi)$, $[L_3]^5(\Pi)$ and $[R_3]^5(\Pi)$ in Example \ref{exam:nonsymLR}}
\label{fig:3}
\end{figure}
\end{exam}

\section{Acknowledgements}
The author is grateful for the financial support from the Commission on Higher Education and the Thailand Research Fund through grant no.~RSA5680037.  We would also like to thank Wolfgang Trutschnig and Anusorn Junchuay for sharing and helping with the R codes used to produce copulas with fractal supports and their copula factors.

\end{document}